\documentclass[12pt,reqno]{amsart}

\usepackage{amsmath,amsfonts,amsthm,amssymb,amscd, dsfont,tipa,mathtools,listings, rotating, multirow}
\usepackage{hyperref}
\usepackage{color}
\usepackage{comment}
\usepackage{seqsplit}
\usepackage{enumitem}
\usepackage[utf8]{inputenc}
\newcommand{\bburl}[1]{\textcolor{blue}{\url{#1}}}

\newcommand\be{\begin{equation}}
\newcommand\ee{\end{equation}}
\newcommand\bi{\begin{itemize}}
\newcommand\ei{\end{itemize}}
\newcommand\ben{\begin{enumerate}}
\newcommand\een{\end{enumerate}}

\newtheorem{theorem}{Theorem}

\newtheorem{proposition}{Proposition}

\theoremstyle{definition}
\newtheorem{definition}{Definition}

\numberwithin{equation}{section}

\begin{document}

\title{An Alternative Approach to Computing $\beta(2k+1)$}
\author{Naomi Tanabe} \address[Naomi Tanabe]{Department of Mathematics, Bowdoin College, Brunswick, ME 04011}\email{ntanabe@bowdoin.edu}
\author{Nawapan Wattanawanichkul} \address[Nawapan Wattanawanichkul]{University of Illinois Urbana-Champaign, Champaign, IL 04011}\email{nawapan2@illinois.edu}

\begin{abstract}
    This paper presents a new approach to evaluating the special values of the Dirichlet beta function, $\beta(2k+1)$, where $k$ is any nonnegative integer. Our approach relies on some properties of the Euler numbers and polynomials, and uses basic calculus and telescoping series. By a similar procedure, we also yield an integral representation of $\beta(2k)$. The idea of our proof adapts from a previous study by Ciaurri et al., where the authors introduced a new proof of Euler's formula for $\zeta(2k)$.
\end{abstract}

\maketitle

\section{Introduction}\label{sec:intro}
It is well known that the value of the Riemann $\zeta$-function at a positive even integer $2k$ can be expressed as
\begin{equation}\label{eq:Euler's formula} 
    \zeta(2k) \ = \ \sum_{n=1}^{\infty}\frac{1}{n^{2k}} \ = \ \cfrac{(-1)^{k-1}2^{2k-1}\pi^{2k}}{(2k)!}B_{2k},
\end{equation}
where $B_{k}$ is the $k$-th Bernoulli number. One of the classical proofs of this formula is attributed to Euler, which involves considering the expansion of $\pi z\cot(\pi z)$ in two different ways. However, over time, numerous other proofs have been developed utilizing a variety of techniques and approaches, including notable examples such as \cite{apostol}, \cite{benyi}--\cite{murty}, \cite{osler}, and \cite{tsumura}--\cite{yue}.
The multitude of proofs reflects the fundamental importance of this formula and the richness of the mathematical concepts it connects. 

On the other hand, the Riemann $\zeta$-function has been generalized in many ways, including the Dirichlet $L$-functions. In a similar manner to Equation \eqref{eq:Euler's formula}, formulas for the special values of the Dirichlet $L$-functions have been established as follows (for details see, for example, \cite[Section 7-2,~Corollary~2.10]{neukirch}).

\begin{theorem}[\cite{neukirch}]
Let $\chi$ be a primitive character of conductor $N$ and $k$ be a positive
integer satisfying $\chi(-1) = (-1)^k$. Then we have
\begin{equation}\label{eq:Euler_chi} L(k, \chi) = (-1)^{k-1} \frac{\tau(\chi)}{2}\left(\frac{2\pi i}{N}\right)^k \frac{B_{k,\overline\chi}}{k!},\end{equation}
where $B_{k,\overline{\chi}}$ is the generalized Bernoulli number associated with the conjugate of the character $\chi$, and $\tau(\chi)$ is the Gauss sum of the character defined as
\begin{equation*}
    \tau(\chi) \ = \ \sum_{a=1}^N \chi(a) e^{\frac{2\pi i a}{N}}.
\end{equation*}
\end{theorem}
These formulas play a critical role in number theory, particularly in the study of primes in arithmetic progressions, and have many connections with various mathematical objects such as modular forms, automorphic representations, and Galois representations.

In this paper, we focus on Equation \eqref{eq:Euler_chi} particularly for the $L$-function associated with the primitive Dirichlet character $\chi_4$ modulo $4$, also known as  the \textit{Dirichlet $\beta$-function} 
\begin{equation*} 
\beta(s)=\sum_{m=0}^\infty \frac{(-1)^m}{(2m+1)^{s}}.
\end{equation*}
More precisely, it is our aim to present an alternative approach to computing $L(2k+1, \chi_4)=\beta(2k+1)$, whose formula can be stated as in the following theorem, according to Equation \eqref{eq:Euler_chi}. 

\begin{theorem}\label{thm:beta_odd}
Let $\chi_4$ be a primitive character of conductor $4$ and $k$ be a nonnegative
integer. We have
\begin{equation}\label{eq:thm_beta_odd}  
  \beta(2k+1)=  (-1)^{k+1}\left(\frac{\pi}{2}\right)^{2k+1} \frac{B_{2k+1,\chi_4}}{(2k+1)!}
 ,\end{equation}
where $B_{2k+1,\chi_4}$ is the $2k+1$-th generalized Bernoulli number associated with $\chi_4$.
\end{theorem}
To prove the above theorem, we use the techniques analogous to those introduced in \cite{ciaurri}. 
The method utilizes basic calculus and telescoping series to derive the desired formula. Applying the same techniques, we also obtain an integral representation for $\beta(2k)$, when $k$ is a positive integer.

\begin{theorem}\label{thm:beta_even}
For any positive integer $k$, we have
\begin{equation}\label{eq:thm_beta_even}  \beta(2k) \ = \ \frac{(-1)^{k-1}\pi^{2k}}{2(2k-1)!}\int_0^{1/2}E_{2k-1}(t)\sec({\pi t})dt.\end{equation}
\end{theorem}

In the next section, we recall definitions and important properties of Euler polynomials $E_k(x)$, which are crucial to our work. Then Section \ref{sec:compute beta} is devoted to our proof of Theorem \ref{thm:beta_odd}, a formula for $\beta(2k+1)$. Lastly, we derive Theorem \ref{thm:beta_even}, an integral representation for $\beta(2k)$, in Section \ref{sec:beta(2k)}.

\section{Euler Numbers and Polynomials}\label{sec:properties}

In this section, we introduce some useful properties of the Euler numbers and polynomials, which will be used repeatedly in Sections \ref{sec:compute beta} and \ref{sec:beta(2k)}.  

\begin{definition}\label{def:Eulernumber} The \textit{$k$-th Euler number} $E_k$ is defined by the generating function
\begin{equation} \label{eq:EulernumGenerating}\sum_{k=0}^{\infty} E_k\frac{t^k}{k!}  \ = \ \frac{2e^{t}}{e^{2t}+1}.\end{equation}  \end{definition} 
By expanding the right-hand side of the equation above, the first few Euler numbers can be observed as 
    $$ E_0 = 1, \ E_1=0, \ E_2= -1, \ E_3=0, \ E_4=5, \ \dots.$$

\begin{definition}\label{def:Eulerpoly} The \textit{$k$-th Euler polynomial $E_k(x)$} is defined by the generating function
\begin{equation} \label{eq:EuPolyGenerating}
\sum_{k=0}^{\infty} E_k(x)\frac{t^k}{k!} \ = \ \frac{2e^{xt}}{e^t+1},\text{ where } |t| \le \pi, x \in \mathbb{R}.\end{equation} \end{definition}
Again, by expanding the right-hand side of the above equation, we see that the first few Euler polynomials are
    $$E_0(x) = 1, \ E_1(x)=x -\frac{1}{2}, \ E_2(x)= x^2-x, \ E_3=x^3 - \frac{3}{2}x^2+\frac{1}{4}, \ \dots .$$

Using Definitions \ref{def:Eulernumber} and \ref{def:Eulerpoly}, we observe the following noteworthy proposition.

\begin{proposition}\label{props:Euler} For the Euler polynomials and $k\in \mathbb{Z}_{\ge 0}$, the followings are true: 
\begin{enumerate}[label=(1.\arabic*)]
  \item \label{prop:E(1-x)} \quad $E_k(1-x)  =  (-1)^kE_k(x)$ and, in particular, $\displaystyle E_{2k+1}\left(1/2\right)  =  0,$ \vspace{0.1in}
  \item \label{prop:powerof2} \quad $\displaystyle E_k  =  2^k E_k\left(1/2\right),$ \vspace{0.1in}
  \item \label{prop:sumx,x+1} \quad $E_k(x+1) + E_k(x) = 2x^k,$ \vspace{0.1in}
  \item \label{prop:eval2k} \quad $E_{2k}(1) = E_{2k}(0)= 0,$ \vspace{0.1in}
  \item \label{prop:Ederivative)} \quad $E'_0(x) = 0$ and $E'_k(x) = kE_{k-1}(x)$ when $k \ge 1,$ \vspace{0.1in}
  \item \label{prop:genBerE2k} \quad $\displaystyle E_{2k}\left(1/2\right) = - \frac{B_{2k+1,\chi_4}}{(2k+1)2^{2k-1}}$, where $B_{2k+1,\chi_4}$ is the generalized Bernoulli number associated with the primitive character modulo 4.
\end{enumerate}
\end{proposition}
\begin{proof} To prove the first statement, we substitute $x$ with $1-x$ in Equation \eqref{eq:EuPolyGenerating} and get
\[\quad \sum_{k=0}^{\infty} E_k(1-x) \frac{t^k}{k!} \ = \ \frac{2e^{(1-x)t}}{e^t+1} \ = \  \frac{2e^{-xt}}{e^{-t}+1} \ = \ \sum_{k=0}^{\infty} (-1)^k E_k(x) \frac{t^k}{k!}.\]
Comparing the coefficients of $t^k$ term on both sides of the equation gives us the desired result. The second part in \ref{prop:E(1-x)} then follows by evaluating the equation at $x= 1/2$ when $k$ is an odd integer.
The second and third statements are obtained similarly, by evaluating  Equation \eqref{eq:EuPolyGenerating} at $x =1/2$ and at $x-1,$ respectively.

The statement \ref{prop:eval2k} follows from substituting $x=0$ to equations in  \ref{prop:E(1-x)} and \ref{prop:sumx,x+1}, which yields $E_{2k}(1)=E_{2k}(0)$ and $E_{2k}(1) + E_{2k}(0) = 0$, respectively.

To verify \ref{prop:Ederivative)}, we differentiate Equation \eqref{eq:EuPolyGenerating} with respect to $x$;
\begin{equation*} \hspace{0.3in}
\sum_{k=0}^{\infty}E'_k(x) \frac{t^k}{k!} \ = \ \frac{d}{dx} \frac{2e^{xt}}{e^t+1} \ = \ t \cdot \frac{2e^{xt}}{e^t+1}.
\end{equation*}
The right-hand side of the equation is then the product of $t$ and the generating function of the the Euler polynomials.  Therefore, we see that
\begin{equation} \qquad
\sum_{k=0}^{\infty}E'_k(x) \frac{t^k}{k!} \ = \   \sum_{k=0}^{\infty}E_k(x) \frac{t^{k+1}}{k!} \ = \ \sum_{k=1}^{\infty}E_{k-1}(x) \frac{t^{k}}{(k-1)!},\label{eqn:der Euler Polynomials}
\end{equation}
which means $E'_k(x) = k E'_{k-1}(x)$ when $k \ge 1$. Since the constant term of the right-hand side of Equation \eqref{eqn:der Euler Polynomials} is 0, we conclude that $E'_0(x) = 0$.

Lastly, for \ref{prop:genBerE2k}, we recall the relation between the $k$-th generalized Bernoulli number $B_{k, \chi}$ associated with $\chi$ and the $k$-th Bernoulli polynomial $B_k(x)$ given by  
\begin{equation*}\displaystyle B_{k,\chi} \ = \ N^{k-1} \sum_{a=1}^N \chi(a) B_k\left(\frac{a}{N}\right).\end{equation*}
Here, $N$ is the conductor of the character $\chi$. See, for example, \cite[Section~4.3]{arakawa}. In particular, when $\chi = \chi_4$, 
\begin{equation}
    B_{k,\chi_4} \ = \ 4^{k-1}\left(B_k\left(\frac{1}{4}\right) - B_k\left(\frac{3}{4}\right)\right).  \label{eqn:gen Bernoulli mod4}
\end{equation} 
Likewise, the Euler polynomials can be related to the Bernoulli polynomials as  \[E_{k-1}(x) \ = \ \frac{2^k}{k} \left(B_k\left(\frac{x+1}{2}\right) - B_k\left(\frac{x}{2}\right)\right),\]
(see \cite{srivastava} for details). 
In particular, when $x = 1/2$, we have
\begin{equation}
E_{k-1}\left(\frac{1}{2}\right) \ = \ \frac{2^k}{k} \left(B_k\left(\frac{3}{4}\right)-B_k\left(\frac{1}{4}\right)\right).  \label{eqn:Srivastava mod4}
\end{equation} 
Comparing Equations \eqref{eqn:gen Bernoulli mod4} and \eqref{eqn:Srivastava mod4} gives us
\begin{equation*}
     E_{k-1}\left(\frac{1}{2}\right) \ = \ \frac{2^k}{k}\cdot \frac{-B_{k,\chi_4}}{4^{k-1}} \ = \ -\frac{B_{k,\chi_4}}{2^{k-2}k}. 
\end{equation*}
A proof is completed by replacing $k$ with $2k+1$.
\end{proof}

\section{Computing \texorpdfstring{$\beta(2k+1)$}{Lg}}\label{sec:compute beta}

In this section, we prove the formula for $\beta(2k+1)$ as stated in Theorem \ref{thm:beta_odd}.

\begin{proof}[Proof of Theorem \ref{thm:beta_odd}]
We will use the following auxiliary integral
\begin{equation}\label{eq:AuxIntegral}
 I(k,m) \ = \ \int_0^{1/2} E_{2k}(t)\sin((2m+1)\pi t) \ dt,
\end{equation} for integers $k,m \ge 0$. For clarity, we split the proof into three main steps.

\textbf{1) Summing auxiliary functions.} First, we find the recurrence relation among the auxiliary functions $I(k, m)$ and derive the closed form solution. We begin with the simplest case when $k=0$. Using the fact that $E_0(t) = 1$ for any real $t$, we have that
\begin{equation} \label{eqn:I(0,m)}
I(0,m)
\ = \ \int_0^{1/2} \sin((2m+1)\pi t)\,dt  \ = \  \frac{1}{(2m+1)\pi}. 
\end{equation}
For $k \ge 1$, we integrate Equation \eqref{eq:AuxIntegral} by parts and obtain
\begin{equation*}
\begin{split}
I(k,m) &= - \left[E_{2k}(t)\frac{\cos((2m+1)\pi t)}{(2m+1)\pi}  \right]_{t=0}^{t=1/2} + \int_0^{1/2}E'_{2k}(t)\frac{\cos((2m+1)\pi t)}{(2m+1)\pi} \, dt\\ 
&= \frac{1}{(2m+1)\pi}\int_0^{1/2}E'_{2k}(t)\cos((2m+1)\pi t) \, dt,
\end{split}
\end{equation*}
where the last equality follows from \ref{prop:eval2k}.

Now, applying \ref{prop:Ederivative)} and integrating by parts again, we get 
\begin{equation*}
\begin{split}
I(k,m) &=  - \frac{2k(2k-1)}{(2m+1)^2\pi^2} \int_0^{1/2}E_{2k-2}(t)\sin((2m+1)\pi t) \, dt\\
  &= \frac{-2k(2k-1)}{(2m+1)^2\pi^2} I(k-1,m). 
\end{split}
\end{equation*}
Applying this recurrence relation repeatedly, together with the value of $I(0, m)$ from Equation $\eqref{eqn:I(0,m)}$, we obtain the closed form of our auxiliary functions
\begin{equation*}
\begin{split}
    I(k,m) &=  \frac{-2k(2k-1)}{(2m+1)^2\pi^2}\cdot  \frac{-(2k-2)(2k-3)}{(2m+1)^2\pi^2} \dots \frac{-2\cdot1}{(2m+1)^2\pi^2}\cdot \frac{1}{(2m+1)\pi} \nonumber \\
    &= \frac{(-1)^k(2k)!}{(2m+1)^{2k+1}\pi^{2k+1}},
\end{split}
\end{equation*} for any nonnegative integers $k$ and $m$.
Multiplying each $I(k, m)$ by $(-1)^m$ and summing up over nonnegative integers $m$ relate $I(k,m)$ to $\beta(2k+1)$ as
\begin{equation}
    \sum_{m=0}^{\infty}(-1)^m I(k,m) \ = \ \sum_{m=0}^{\infty} \frac{(-1)^m(-1)^k(2k)!}{(2m+1)^{2k+1}\pi^{2k+1}} \ = \ \frac{(-1)^k(2k)!}{\pi^{2k+1}} \beta(2k+1). \label{eqn:sumI+L}
\end{equation}

\textbf{2) Modifying auxiliary functions.} We now modify our auxiliary functions $I(k, m)$ as follows:
\begin{equation*}
    I^*(k,m) \ := \ \int_0^{1/2} E^*_{2k}(t) \sin((2m+1)\pi t) \,dt,
\end{equation*}
where 
\begin{equation*} E^*_{2k}(t) \ := \ E_{2k}(t)-\frac{E_{2k}}{2^{2k}}\sin(\pi t). \end{equation*}
This can also be written as 
\begin{equation*}
\begin{split}
I^*(k,m) &= \int_0^{1/2} \left(E_{2k}(t)-\frac{E_{2k}}{2^{2k}}\sin(\pi t)\right) \sin((2m+1)\pi t) \,dt  \\
&= I(k,m) - \int_0^{1/2}\frac{E_{2k}}{2^{2k}}\sin(\pi t)\sin((2m+1)\pi t)\,dt,
\end{split}
\end{equation*}
and therefore, 
\begin{equation}
    I(k,m) \ = \ I^*(k,m) + \frac{E_{2k}}{2^{2k}}\int_0^{1/2}\sin(\pi t)\sin((2m+1)\pi t)\,dt. \label{eq:relateI,Istar}
\end{equation}
Furthermore, applying the following trigonometric identity
\begin{equation*}
    \sin(\alpha)\sin(\beta) \ = \ \frac{\cos(\alpha-\beta) - \cos(\alpha+\beta)}{2}
\end{equation*}
to the integrand in Equation \eqref{eq:relateI,Istar} yields that
\begin{equation*}
\begin{split}
    \frac{E_{2k}}{2^{2k}}\int_0^{1/2}\sin(\pi t)\sin((2m+1)\pi t)\,dt &= \frac{E_{2k}}{2^{2k}}\int_0^{1/2}\frac{\cos(2m\pi t) - \cos((2m+2)\pi t)}{2}\,dt \\
     &=  \begin{cases}
        \displaystyle \frac{E_{2k}}{2^{2k+2}} \hspace{0.24in} \text{if } m=0,\\
        0 \hspace{0.53in} \text{otherwise.}
    \end{cases}
\end{split}
\end{equation*}
We also note that \ref{prop:powerof2} and \ref{prop:genBerE2k} give
\begin{equation*}
    \frac{E_{2k}}{2^{2k+2}}  \ = \  -\cfrac{B_{2k+1,\chi_4}}{(2k+1)2^{2k+1}}.
\end{equation*}
Hence, Equation \eqref{eq:relateI,Istar} can be written as
\begin{equation*}
    I(k,m) \ = \ 
\begin{cases}
    I^*(k,m) -\cfrac{B_{2k+1,\chi_4}}{(2k+1)2^{2k+1}},& \text{if } m = 0,\\
    I^*(k,m),& \text{if } m \ge 1.
\end{cases} 
\end{equation*}
Thus,
\begin{align}
    \sum_{m=0}^{\infty}(-1)^mI(k,m)\nonumber 
    &= \left(I^*(k,0) - \cfrac{B_{2k+1,\chi_4}}{(2k+1)2^{2k+1}}\right)+ \sum_{m=1}^{\infty}(-1)^mI^*(k,m) \nonumber\\
    &= \sum_{m=0}^{\infty}(-1)^mI^*(k,m)-\cfrac{B_{2k+1,\chi}}{(2k+1)2^{2k+1}}.
\label{eq:showsumI*=0}
\end{align}

Comparing this with Equation \eqref{eqn:sumI+L}, it boils down to simplify Equation \eqref{eq:showsumI*=0} to obtain the desired result. Indeed, we will show that $\sum_{m=0}^{\infty}(-1)^mI^*(k,m)=0$ in the following subsection.

\textbf{3) Computing telescoping series.}
We now show that the infinite series $\sum_{m=0}^{\infty}(-1)^m I^*(k,m),$ which is defined as  $\lim_{N\to\infty} \sum_{m=0}^{N}(-1)^m I^*(k,m)$, converges to 0 by using trigonometric identities and telescoping sums. Consider
\begin{equation*}
\begin{split}
&\lim_{N\to\infty} \sum_{m=0}^{N}(-1)^mI^*(k,m)\\
&= \lim_{N\to\infty} \left(I^*(k,0) - I^*(k,1)+ \cdots +(-1)^{N-1}I^*(k,N-1)+(-1)^NI^*(k,N) \right)\\
&= \lim_{N\to\infty} \int_0^{1/2} \Big(E^*_{2k}(t)\sin(\pi t) -  E^*_{2k}(t)\sin(3\pi t) + \cdots\\
& \hspace{0.4in} +(-1)^{N-1} E^*_{2k}(t)\sin((2N-1)\pi t) +(-1)^N E^*_{2k}(t)\sin((2N+1)\pi t) \Big) \,dt. \end{split}
\end{equation*}
Applying the following trigonometric identity 
\begin{equation*}
    \sin((2m+1)x) \ = \ \cfrac{\cos((2m-1)x)-\cos((2m+3)x)}{2\sin(2x)},
\end{equation*}
we obtain the telescoping series
\begin{align}\label{eqn:telescopic} 
& \displaystyle \lim_{N \to \infty}\int_0^{1/2}  \Big( E^*_{2k}(t)\cdot\cfrac{\cos(-\pi t)  -  \cos(3\pi t)}{2\sin(2\pi t)}  -  E^*_{2k}(t)\cdot\cfrac{\cos(\pi t)  -  \cos(5\pi t)}{2\sin(2\pi t)} \nonumber \\
& \hspace{0.5in}  +  E^*_{2k}(t)\cdot\cfrac{\cos(3\pi t)-\cos(7\pi t)}{2\sin(2\pi t)}  -  E^*_{2k}(t)\cdot\cfrac{\cos(5\pi t)- \cos(9\pi t)}{2\sin(2\pi t)} +  \cdots  \nonumber \\
&\hspace{0.5in}  +  (-1)^{N} E^*_{2k}(t)\cdot\cfrac{\cos((2N-1)\pi t)-\cos((2N+3)\pi t)}{2\sin(2\pi t)} \Big)  \,dt. 
\end{align}
To cancel out repetitive terms in Equation \eqref{eqn:telescopic}, we need to extend the function  \[f(t) \ = \ \cfrac{E^*_{2k}(t)}{\sin(2 \pi t)}, \qquad \text{ for } t\in (0, 1/2),\]
to $t= 0$ and $1/2$.
 
When $t=0$, we note that $E^*_{2k}(0) = E_{2k}(0) - \frac{E_{2k}}{2^{2k}}\cdot \sin( 0) =0$ by \ref{prop:eval2k}. We then evaluate the limit of $f(t)$ when $t$ approaches $0$ using L'H\^opital's rule as follows
\begin{equation*}
\lim_{t\to 0} \frac{E_{2k}(t) - \cfrac{E_{2k}}{2^{2k}}\sin(\pi t)}{\sin(2\pi t)}
\ = \ \frac{2k E_{2k-1}(0)-\cfrac{E_{2k}}{2^{2k}}\pi}{2\pi},
\end{equation*}
which is some constant.

As for $t=1/2$, notice that $E^*_{2k}(1/2) = E_{2k}(1/2) - \frac{E_{2k}}{2^{2k}}\sin(\pi/2) =0$ by  \ref{prop:powerof2}. Then the limit of $f(t)$ as $t$ approaches $1/2$ can be evaluated as 
\begin{equation*}
\lim_{t\to1/2} \frac{E_{2k}(t) - \cfrac{E_{2k}}{2^{2k}}\sin(\pi t)}{\sin(2\pi t)}
\ = \ \frac{2k\cdot E_{2k-1}(1/2)- \cfrac{E_{2k}}{2^{2k}}\pi\cos(\pi/2)}{2\pi \cos(\pi)},
\end{equation*}
which equals $0$ by using \ref{prop:E(1-x)}. Thus, $f(t)$ is well-defined on $[0,1/2]$, and, hence, 
most of the terms in Equation \eqref{eqn:telescopic} get cancelled. Moreover, since the first two terms $f(t)\cos(-\pi t)$ and $f(t)\cos(\pi t)$ are equal, we are left with
\begin{align}\label{eq:even case}
    &\sum_{m=0}^{\infty}(-1)^mI^*(k,m) \nonumber  \\
    &= \displaystyle \lim_{N \to \infty}(-1)^{N-1} \int_0^{1/2} \frac{E^*_{2k}(t)}{2\sin(2\pi t)}\left(\cos((2N+1)\pi t)-\cos((2N+3)\pi t)\right)\,dt  \nonumber \\
    &= \displaystyle \lim_{N \to \infty}(-1)^{N-1} \int_0^{1/2} \frac{E^*_{2k}(t)}{2\sin(2\pi t)}\left(-2\sin((2N+2)\pi t)\sin(-\pi t)\right)\,dt  \nonumber \\
    &= \displaystyle \lim_{N \to \infty}(-1)^{N-1}\int_0^{1/2} \frac{E^*_{2k}(t)}{2\cos(\pi t)}\left(\sin((2N+2)\pi t)\right)\,dt.
\end{align}

To proceed further, we justify that the function $\frac{E^*_{2k}(t)}{2\cos(\pi t)}\sin((2N+2)\pi t)$ is differentiable on $[0,1/2]$ with continuous derivative. Similar to the case of $f(t)$, we extend the function \begin{equation*}
g(t) \ = \ \cfrac{E^*_{2k}(t)}{\cos(\pi t)}, \hspace{0.2in} \text{ for } t\in [0, 1/2),
\end{equation*} to $t=1/2$, which can be achieved by applying  \ref{prop:E(1-x)} and \ref{prop:Ederivative)}:
\begin{equation*}
\lim_{t\to1/2} \frac{E_{2k}(t) - \cfrac{E_{2k}}{2^{2k}}\sin(\pi t)}{2\cos(\pi t)}
\ = \ 0. \end{equation*}
Therefore, $g(t)$ is differentiable with continuous derivative on $[0,1/2]$.

We now consider the integral on the right-hand side of the last equation of \eqref{eq:even case}. Writing $(2N+2)\pi = R$ and integrating by parts give 
\begin{equation*}
\int_0^{1/2}g(t)\sin(Rt) \,dt \ = \ -\cfrac{\cos(R/2)}{R} g(1/2) + \frac{1}{R} g(0) + \int_0^{1/2}g'(t)\cfrac{\cos(Rt)}{R}\,dt.\\
\end{equation*}
The boundedness of $g(0)$, $g(1/2)$ and $g'(t)$ shows that each term in the above sum approaches zero as $R$ approaches infinity, and therefore \[\lim_{N\to \infty} \sum_{m=0}^{N}(-1)^mI^*(k,m)=0.\] 
Thus, Equation \eqref{eq:showsumI*=0} is simplified as
$$ \sum_{m=0}^{\infty}(-1)^mI(k,m) \ = \  - \cfrac{B_{2k+1,\chi}}{(2k+1)2^{2k+1}}.$$ 
This, together with Equation \eqref{eqn:sumI+L},  completes the proof. \end{proof}

\section{An Integral Representation of \texorpdfstring{$\beta(2k)$}{Lg}}\label{sec:beta(2k)}

This section is devoted to obtaining the  integral representation of $\beta(2k)$ as stated in Theorem \ref{thm:beta_even}. In this case, we split the proof into two steps.

\begin{proof}[Proof of Theorem \ref{thm:beta_even}] We consider a slightly different auxiliary integrals 
\begin{equation} \label{defi:J(k,m)}
J(k,m) \ = \ \int_{0}^{1/2}E_{2k+1}(t)\cos((2m+1)\pi t)\,dt
\end{equation}
with integers $k,m \ge 0$. 

\textbf{1) Summing auxiliary functions.}
Similar to the case of $I(0,m)$, applying \ref{prop:Ederivative)} and integrating by parts yield
\begin{equation*}
J(0,m)
\ = \ \frac{E_1\left(\frac{1}{2}\right)\sin\left(\frac{(2m+1)}{2}\pi\right)- E_1(0)\sin(0)}{(2m+1)\pi}  -  \int_0^{1/2}E_0(t)\frac{\sin((2m+1)\pi t)}{(2m+1)\pi}\,dt.
\end{equation*}
Then, using \ref{prop:E(1-x)}, together with the facts that $E_{0}(t) = 1$ and $\sin(0) = 0$, we are left with
\begin{equation}\label{eq:J(0,m)}
J(0,m) \ = \  -  \int_0^{1/2}\frac{\sin((2m+1)\pi t)}{(2m+1)\pi}\,dt
\ = \  - \frac{1}{(2m+1)^2\pi^2}.
\end{equation}

Now we consider Equation \eqref{defi:J(k,m)} when $k \ge 1$. Integrating by parts twice, along with \ref{prop:E(1-x)}, \ref{prop:eval2k}, and \ref{prop:Ederivative)}, gives us
\begin{equation} \label{eq:recurJ(k,m)}
J(k,m)
= - \frac{(2k+1)(2k)}{(2m+1)^2\pi^2}J(k-1,m).
\end{equation}

Putting Equations \eqref{eq:J(0,m)} and \eqref{eq:recurJ(k,m)} together provides the closed form of $J(k, m)$ as
\begin{equation*}\label{eq:closedJ(k,m)}
J(k,m) 
\ = \ \frac{(-1)^{k+1}(2k+1)!}{(2m+1)^{2k+2}\pi^{2k+2}}.
\end{equation*}
Therefore, we can relate $J(k,m)$ to $\beta(2k)$ as
\begin{equation*}
    \sum_{m=0}^{\infty}(-1)^mJ(k-1,m) \ = \ \sum_{m=0}^{\infty}(-1)^m\frac{(-1)^{k}(2k-1)!}{(2m+1)^{2k}\pi^{2k}} \ = \  \frac{(-1)^{k}(2k-1)!}{\pi^{2k}} \beta(2k), 
\end{equation*}
for any $k \ge 1$, or equivalently, \begin{equation}
    \label{eq:betaJ} \beta(2k) \ = \ \frac{(-1)^{k}\pi^{2k}}{(2k-1)!}\sum_{m=0}^{\infty}(-1)^mJ(k-1,m).
\end{equation}

\textbf{2) Computing telescoping series.} We now simplify the right-hand side of Equation \eqref{eq:betaJ} by exploiting a telescoping series and the  trigonometric identity 
\begin{equation}
     \cos((2m+1)\pi t) \ = \ \frac{\cos(2m\pi t) + \cos((2m+2)\pi t)}{2\cos(\pi t)}.\label{eq:trigJ} \end{equation} 
Applying Equation \eqref{eq:trigJ} in Equation \eqref{defi:J(k,m)}, we obtain that
\begin{equation*}
J(k-1,m) \ = \ \int_{0}^{1/2}E_{2k-1}(t)\frac{\cos(2m\pi t) + \cos((2m+2)\pi t)}{2\cos(\pi t)} \ dt. 
\end{equation*}
Multiplying each $J(k-1, m)$ by $(-1)^m$ and summing up over nonnegative integers $m$ give 
\begin{equation*}
\begin{split}
& \sum_{m=0}^{\infty}(-1)^m J(k-1,m)  \\
&\hspace{.2in} =\lim_{N\to\infty}\int_{0}^{1/2}\left(E_{2k-1}(t)\frac{\cos(0) + \cos(2\pi t)}{2\cos(\pi t)}-\cdots \right. \\
&\hspace{1.5in} \left. +(-1)^N E_{2k-1}(t)\frac{\cos(2N \pi t)+\cos((2N+2)\pi t)}{2\cos(\pi t)}\right) \ dt.
\end{split}
\end{equation*}
Then we extend the function
$$h(t) \ := \ \frac{E_{2k-1}(t)}{2\cos(\pi t)}, \hspace{.3in} t\in [0, 1/2),$$ 
in each summand, 
to $t = 1/2$ by continuity. By applying L'H\^opital's rule, we obtain
\begin{equation*} 
\lim_{t \to 1/2} \frac{E_{2k-1}(t)}{2\cos(\pi t)} \ = \ \frac{(2k-1)E_{2k-2}(1/2)}{-2\pi},
\end{equation*}
which is some constant. Thus, $h(t)$ is differentiable on $[0,1/2]$ with a continuous derivative, and, hence, most of the terms in the above integral now get cancelled, leaving us
\begin{align}\label{eq:telescopic sum2}
&\sum_{m=0}^{\infty}(-1)^m J(k-1,m)  \nonumber  \\
&= \lim_{N \to \infty}\int_{0}^{1/2}\left(E_{2k-1}(t)\frac{\cos(0)}{2\cos(\pi t)} +(-1)^N E_{2k-1}(t)\frac{\cos((2N+2)\pi t)}{2\cos(\pi t)}\right) \,dt \nonumber \\ 
&= \int_{0}^{1/2}\frac{E_{2k-1}(t)\sec(\pi t)}{2} \,dt + \lim_{N \to \infty}(-1)^N \int_{0}^{1/2} E_{2k-1}(t)\frac{\cos((2N+2)\pi t)}{2\cos(\pi t)} \,dt.
\end{align}
To further simplify the above expression, we will show that, in fact, 
\begin{equation}
    \label{eq:nullexpression}  \lim_{N\to\infty}(-1)^N \int_{0}^{1/2} E_{2k-1}(t)\frac{\cos((2N+2)\pi t)}{2\cos(\pi t)}dt \nonumber 
    \end{equation}
is null.  
Let $R$ denote $(2N+2)\pi$. The integral above then equals
\begin{equation*}
\begin{split}
& \lim_{N\to\infty}(-1)^N \int_{0}^{1/2} h(t)\cos(R t)dt\\
&= \lim_{N\to\infty}(-1)^N \left(h(1/2)\frac{\sin(R/2)}{R} - h(0)\frac{\sin(0)}{R} - \int_{0}^{1/2}h'(t)\frac{\sin(R t)}{R}dt\right).
\end{split}
\end{equation*}
Since $h(1/2), h(0)$, and $h'(t)$ are bounded, each summand approaches $0$ as $R \to \infty$, and therefore this limit is indeed 0. Thus, the summation \eqref{eq:telescopic sum2} is 
\begin{equation*}
    \sum_{m=0}^{\infty}(-1)^mJ(k-1,m) \ = \ \int_{0}^{1/2}\frac{E_{2k-1}(t)\sec(\pi t)}{2}dt.
\end{equation*}
Substituting this back into Equation \eqref{eq:betaJ} yields 
\[\beta(2k) \ = \ \frac{(-1)^{k-1}\pi^{2k}}{2(2k-1)!}\int_0^{1/2}E_{2k-1}(t)\sec({\pi t})dt,\]
for all positive integers $k$, as desired.
\end{proof}

\vspace{0.2in}

\noindent {\bf Acknowledgement.} We are grateful for the support from the Kibbe Science Fellowship from Bowdoin College.

\end{document}